\documentclass[10pt]{article}

\usepackage{hyperref}
\usepackage{xcolor}
\usepackage{amssymb}
\usepackage{amsfonts}
\usepackage{amsthm}
\usepackage{amsmath}
\usepackage{float}
\usepackage{bigints}
\usepackage{fullpage}
\usepackage{mathtools}
\usepackage[utf8]{inputenc} 
\usepackage{cancel}
\usepackage{verbatim}

\newtheorem{theorem}{Theorem}

\newtheorem{lemma}[theorem]{Lemma}
\newtheorem{sublemma}{Sublemma}

\newtheorem{proposition}[theorem]{Proposition}

\newtheorem{definition}{Definition}
\newtheorem{problem}[theorem]{Problem}

\newtheorem{remark}[theorem]{Remark}

\newcommand{\inter}{\rm int}

\newcommand{\x}{\mathbf{x}}
\newcommand{\p}{\mathbf{p}}
\newcommand{\q}{\mathbf{q}}
\newcommand{\y}{\mathbf{y}}

\newcommand{\oo}{\mathbf{o}}

\newcommand{\K}{\mathbf{K}}

\newcommand{\PP}{\mathbf{P}}

\newcommand{\Ee}{\mathbb{E}}

\newcommand{\Ss}{\mathbb{S}}
\newcommand{\Mm}{\mathbb{M}}

\newcommand{\Rr}{\mathbb{R}}

\newcommand{\Ed}{\Ee^d}
\newcommand{\Rd}{\Rr^d}

\newcommand{\Md}{\Mm^d}

\newcommand{\noshow}[1]{}

\title{On $k$-diametral point configurations in Minkowski spaces
\footnote{\hspace*{.35cm} 2010 Mathematics Subject Classification: 52A20, 52A21, 52C35.}}
\author{K\'{a}roly Bezdek\thanks{Partially supported by a Natural Sciences and 
Engineering Research Council of Canada Discovery Grant.} and Zsolt L\'angi\thanks{Partially supported by the National Research, Development and Innovation Office, NKFI, K-119670, the J\'anos Bolyai Research Scholarship of the Hungarian Academy of Sciences, and grants BME FIKP-V\'IZ and \'UNKP-19-4 New National Excellence Program by the Ministry of Innovation and Technology.}
}

\date{}

\begin{document}

\maketitle

\begin{abstract}
The structure of $k$-diametral point configurations in Minkowski $d$-space is shown to be closely related to the properties of $k$-antipodal point configurations in $\Rd$. In particular, the maximum size of $k$-diametral point configurations of Minkowski $d$-spaces is obtained for given $k\geq 2$ and $d\geq 2$ generalizing Petty's results (Proc. Am. Math. Soc. 29: 369-374, 1971) on equilateral sets in Minkowski spaces. Furthermore, bounds are derived for the maximum size of $k$-diametral point configurations in Euclidean $d$-space. In the proofs convexity methods are combined with volumetric estimates and combinatorial properties of diameter graphs.
\end{abstract}

\section{Introduction}\label{sec:intro}

Let $\K_{\oo}\subset \Rd$ be an $\oo$-symmetric convex body, i.e., a compact convex set with nonempty interior symmetric about the origin $\oo$ in $\Rd$. Let ${\cal{K}}_{\oo}^{d}$ denote the family of $\oo$-symmetric convex bodies in $\Rd$.  Moreover, let $\| \cdot \|_{\K_{\oo}}$ denote the norm generated by $\K_{\oo}\in {\cal{K}}_{\oo}^{d}$, which is defined by $\|\x\|_{\K_{\oo}}:=\min\{\lambda \geq 0\  |\  \x\in \lambda\K_{\oo}  \}$ for $\x\in \Rd$. Furthermore, let us denote $\Rd$ with the norm $\| \cdot \|_{\K_{\oo}}$ by $\Md_{\K_{\oo}}$ and call it the {\it Minkowski space of dimension $d$ generated by} $\K_{\oo}$. The following definition introduces the central notion for our paper.

\begin{definition}\label{k-diametral}
We call the labeled point set $X:=\{\x_1,\x_2, \dots , \x_n\}\subset\Rd$ a {\rm point configuration of $n$ points} in $\Rd$, where $n\geq 1$ and $d\geq 2$. Here
the points $\x_1,\x_2, \dots , \x_n$ are not necessarily all distinct and therefore $n$ is not necessarily equal to the number of distinct points in $X$. Next, let $X:=\{\x_1,\x_2, \dots , \x_n\}\subset\Rd$ be a point configuration of $n$ points with some positive diameter in $\Md_{\K_{\oo}}$, i.e., let  ${\rm diam}_{\Md_{\K_{\oo}}}(X):=\max\{\|\x_i-\x_j\|_{\K_{\oo}}\ |\ 1\leq i< j\leq n\}>0$. Let $k\geq 2$ be an integer. Then we say that $X:=\{\x_1,\x_2, \dots , \x_n\}\subset\Rd$ is a {\rm $k$-diametral point configuration of $n$ points in $\Md_{\K_{\oo}}$} if any $k$-tuple $\x_{n_1}, \x_{n_2}, \dots , \x_{n_k}$, $1\leq n_1< n_2 < \dots  < n_k\leq n$ chosen from $X$ contains two diametral points, i.e., it contains two points say, $\x_{n_i}$ and $\x_{n_j}$, $1\leq i<j\leq k$ such that $\| \x_{n_i}- \x_{n_j}\|_{\K_{\oo}} ={\rm diam}_{\Md_{\K_{\oo}}}(X)$. In particular, a $2$-diametral point configuration is called a {\rm diametral} or simply an {\rm equilateral} point configuration, and a $3$-diametral point configuration is called an {\rm almost diametral} point configuration. Finally, let us denote the largest $n$ for which there exists a $k$-diametral point configuration of $n$ points in $\Md_{\K_{\oo}}$, by $f_k(\Md_{\K_{\oo}})$ and call it the {\rm $k$-diametral number} of point configurations in $\Md_{\K_{\oo}}$.
\end{definition}

\begin{remark}\label{Petty}
We note that Petty \cite{Pe} proved the inequality $f_2(\Md_{\K_{\oo}})\leq 2^d$ for all $d\geq 2$ and $\K_{\oo}\in {\cal{K}}_{\oo}^{d}$. Moreover, he has shown that $f_2(\Md_{\K_{\oo}})=2^d$ if and only if $\K_{\oo}$ is an affine $d$-cube of $\Rd$ in which case every equilateral point configuration of $2^d$ points is identical to the vertex set of a homothetic affine $d$-cube.
\end{remark}

\begin{remark}\label{Polyanskii}
Let $\Ed$ denote the $d$-dimensional Euclidean space. Recall that Polyanskii \cite{Po} calls a subset of $\Ed$ an {\rm almost-equidistant diameter set} if it has diameter one and if among any three distinct points of the subset some two are at unit distance apart. Thus, any almost-equidistant diameter set of $\Ed$ is an almost diametral point configuration of $\Ed$ in the sense of Definition \ref{k-diametral} but, not necessarily the other way around. As the elegant algebraic method of \cite{Po} extends to point configurations of $\Ed$, therefore the upper bound $2d+4$ proved in \cite{Po} works for $f_3(\Ed)$ as well, i.e.,  $f_3(\Ed)\leq 2d+4$ holds for all $d\geq 2$. Moreover, by taking the vertex set of a regular $d$-simplex in $\Ed$ with each vertex having multiplicity two, one obtains an almost diametral point configuration of $2d+2$ points in $\Ed$. Thus, one may wonder whether $f_3(\Ed)=2d+2$ holds for all $d\geq 2$. We note that Part {\bf (iii)} of Propositon \ref{basic-estimates-Part1} gives a positive answer to this question for $d=2,3$. (See also Remark~\ref{Alon}.)
\end{remark}

The problem of estimating (resp., computing) $f_k(\Md_{\K_{\oo}})$ seems to be a difficult question in general. On the other hand, one can connect this question to other important problems of geometry and obtain some basic estimates for $f_k(\Md_{\K_{\oo}})$. Next, we introduce the definitions and relevant results needed and then, we state those basic estimates for $f_k(\Md_{\K_{\oo}})$ in Proposition \ref{basic-estimates-Part1}.

\begin{definition}\label{Borsuk-number}
Let $b(\Md_{\K_{\oo}})$ denote the smallest positive integer $m$ such that any finite set $Y\subset \Md_{\K_{\oo}}$ with diameter ${\rm diam}_{\Md_{\K_{\oo}}}(Y)>0$ can be partitioned into $m$ sets each having diameter smaller than ${\rm diam}_{\Md_{\K_{\oo}}}(Y)$ in $\Md_{\K_{\oo}}$. We call $b(\Md_{\K_{\oo}})$ the {\rm Borsuk number} of the Minkowski space $\Md_{\K_{\oo}}$.
\end{definition}
 
\begin{definition}\label{homothetic-covering-number}
Let $h(\Rd)$ denote the smallest positive integer $l$ such that the convex hull ${\rm conv}(Y)$ of any finite set $\emptyset\neq Y\subset \Rd$ can be covered by $l$ smaller positive homothetic copies, i.e., there exist $0<\lambda_i<1$, $\y_i\in \Rd$ for $1\leq i\leq l$ such that ${\rm conv}(Y)\subseteq \bigcup_{i=1}^{l}\left(\y_i+\lambda_i{\rm conv}(Y)\right)$. We call $h(\Rd)$ the {\rm Hadwiger number} of $\Rd$.
\end{definition}

\noindent It is well known that Borsuk \cite{Bo} asked whether $b(\Ed)=d+1$ for any $d>2$ and Hadwiger \cite{Ha} conjectured that if $d>2$, then $h(\Rd)=2^d$. Both problems have become longstanding open problems in geometry. For recent surveys on the status of these problems we refer the interested reader to \cite{BeKh}, \cite{CaSc}, and \cite{Ra}. Here we recall the following results that give the core upper estimates for Proposition \ref{basic-estimates-Part1}. On the one hand, Lassak \cite{La} has proved that $b(\Ed)\leq 2^{d-1}+1$ for any $d>1$. On the other hand, Schramm \cite{Sch} proved the inequality $b(\Ed)\leq 5d\sqrt{d}(4+\ln d)\left(\frac{3}{2}\right)^{\frac{d}{2}}$ for all $d>1$. Furthermore, Rogers \cite{Ro} (see also \cite{RoZo}) has proved that $h(\Rd)\leq \binom{2d}{d}d(\ln d+\ln\ln d +5)=O(4^d\sqrt{d}\ln d)$ holds for any $d>1$. Lassak \cite{Las} improved this upper bound of Rogers for some small values of $d$ by showing that $h(\Rd)\leq (d+1)d^{d-1}-(d-1)(d-2)^{d-1}$ for any $d>1$. Finally, just very recently Huang, Slomka, Tkocz, and Vritsiou \cite{HSTV} improved the upper bound of Rogers for sufficiently large values of $d$ by showing that there exist universal constants $c_1,c_2>0$ such that for all $d>1$, one has $h(\Rd)\leq c_14^de^{-c_2\sqrt{d}}$.

\begin{proposition}\label{basic-estimates-Part1}
\item{\bf (i)} For $d\geq 2$, $k\geq 2$, and $\K_{\oo}\in {\cal{K}}_{\oo}^{d}$, one has 
\begin{equation}\label{Minkowski-bounds-1}
f_2(\Md_{\K_{\oo}})\leq \frac{1}{k-1}f_k(\Md_{\K_{\oo}})\leq b(\Md_{\K_{\oo}})\leq h(\Rd)
\end{equation}
\begin{equation}\label{Minkowski-bounds-2}
\leq \min{\left\{\binom{2d}{d}d(\ln d+\ln\ln d +5), (d+1)d^{d-1}-(d-1)(d-2)^{d-1}, c_14^de^{-c_2\sqrt{d}} \right\}}.
\end{equation}

\item{\bf (ii)} $f_k(\Mm^2_{\K_{\oo}})=(k-1)f_2(\Mm^2_{\K_{\oo}})$ holds for all $k\geq 2$ and $\K_{\oo}\in {\cal{K}}_{\oo}^{2}$.

\item{\bf (iii)} $f_k(\Ed)=(k-1)(d+1)$ holds for all $k\geq 3$ and $d=2,3$ (and for $k=2$ and all $d\geq 1$). Moreover, if $k\geq 3$ and $d\geq 4$, then
\begin{equation}\label{Euclidean-bounds-1} 
(k-1)(d+1)\leq f_k(\Ed)\leq (k-1)b(\Ed)
\end{equation}
\begin{equation}\label{Euclidean-bounds-2}
\leq (k-1)\min\left\{2^{d-1}+1, 5d\sqrt{d}(4+\ln d)\left(\frac{3}{2}\right)^{\frac{d}{2}}\right\}.
\end{equation} 

\end{proposition}

\begin{remark}\label{order-estimate}
We note that $5d\sqrt{d}(4+\ln d)\left(\frac{3}{2}\right)^{\frac{d}{2}}\leq2^{d-1}+1$ holds if and only if $d\geq 18$. Moreover, for every $\epsilon>0$ if $d$ is sufficiently large, then $5d\sqrt{d}(4+\ln d)\left(\frac{3}{2}\right)^{\frac{d}{2}}<\left(\sqrt{1.5}+\epsilon\right)^d=(1.224...+\epsilon)^d$.
\end{remark}

\begin{remark}
Following \cite{HL} and \cite{LN}, we may define the \emph{$n$-fold Borsuk number} $b_n(\Md_{\K_{\oo}})$ of the Minkowski space $\Md_{\K_{\oo}}$ as
the smallest positive integer $m$ such that any finite set $Y\subset \Md_{\K_{\oo}}$ with diameter ${\rm diam}_{\Md_{\K_{\oo}}}(Y)>0$ can be covered $n$-fold by $m$ sets having diameters smaller than ${\rm diam}_{\Md_{\K_{\oo}}}(Y)$ in $\Md_{\K_{\oo}}$. It may be worth noting that the proof of Part {\bf (i)} of Proposition \ref{basic-estimates-Part1} yields also the refined inequalities $\frac{1}{k-1}f_k(\Md_{\K_{\oo}})\leq \frac{1}{n} b_n(\Md_{\K_{\oo}})\leq b(\Md_{\K_{\oo}})$ replacing $\frac{1}{k-1}f_k(\Md_{\K_{\oo}})\leq b(\Md_{\K_{\oo}})$ in (\ref{Minkowski-bounds-1}) for all $n\geq 1$.
\end{remark}



\begin{remark}\label{Alon}
Let $q=p^m$ be a power of a prime number $p$, $n=4q-2$ and $d = \binom{n}{2}$. Alon \cite{Alon} proved (see also \cite{book}) that with this choice there is a set $ \subseteq \{ -1, 1 \}^d \subset \Ed$ containing $2^{n-2}$ points such that any subset of $S$ whose diameter is less than that of $S$ contains at most $\sum_{i=0}^{q-2} \binom{n-1}{i}$ points of $S$. He used this result to show that if $d$ is sufficiently large, then $b(\Ed) \geq (1.2)^{\sqrt{d}}$. Nevertheless, using the estimates $\sum_{i=0}^{q-2} \binom{n-1}{i} \leq \frac{4d}{5e} \left( \frac{256}{27}\right)^{\sqrt{\frac{d}{5}}}$ and $2^{n-2} \geq 16^{\sqrt{\frac{d}{5}}-1}$ (see \cite{book}), and the Bertrand-Chebyshev theorem on prime numbers, from his example it follows that if $d$ is sufficiently large, then for any $k \geq 2.7345^{\sqrt{d}}$ there is a $k$-diametral point configuration of at least $3.455^{\sqrt{d}}$ points in $\Ee^d$. This result shows that the natural question whether $f_k(\Ed)=(k-1)(d+1)$ holds for all $k\geq 3$ and $d\geq 4$ has a negative answer.
\end{remark}

Our next theorem, Theorem \ref{basic-estimates-Part2} below, generalizes Petty's results (Theorems 1,2, and 4 of \cite{Pe}) on equilateral sets in Minkowski spaces. It complements the results of Proposition \ref{basic-estimates-Part1} by finding and characterizing $\max_{\K_{\oo}\in {\cal{K}}_{\oo}^{d}} f_k(\Md_{\K_{\oo}})$. The details are as follows.

\begin{definition}\label{k-antipodal}
Let $X:=\{\x_1,\x_2, \dots , \x_n\}\subset\Rd$ be a point configuration of $n$ points in $\Rd$, where $n\geq 1$ and $d\geq 2$. Let $k\geq 2$ be an integer. We say that $X$ is a {\rm $k$-antipodal point configuration in $\Rd$} if any $k$-tuple $\x_{n_1}, \x_{n_2}, \dots , \x_{n_k}$, $1\leq n_1< n_2 < \dots  < n_k\leq n$ chosen from $X$ contains two antipodal points, i.e., it contains two points say, $\x_{n_i}$ and $\x_{n_j}$, $1\leq i<j\leq k$ lying on distinct parallel supporting hyperplanes of the convex hull ${\rm conv} (X)$ of $X$ in $\Rd$. In particular, a $2$-antipodal point configuration is called an {\rm antipodal} point configuration, and a $3$-antipodal point configuration is called an {\rm almost antipodal} point configuration. Finally, let us denote the largest $n$ for which there exists a $k$-antipodal point configuration of $n$ points in $\Rd$, by $F_k(d)$ and call it the {\rm $k$-antipodal number} of point configurations in $\Rd$.
\end{definition}

\begin{remark}\label{Danzer-Grunbaum}
Recall that according to Danzer and Gr\"unbaum \cite{DaGr}, $F_2(d)=2^d$ for all $d\geq 2$. Furthermore, their volumetric method combined with an earlier result of Groemer \cite{Gr} (on tiling a convex body into homothetic convex bodies) implies that if $X$ is an antipodal point configuration of $2^d$ points in $\Rd$, then $X$ must be identical to the vertex set of an affine $d$-cube.
\end{remark}

\begin{theorem}\label{basic-estimates-Part2}
The point configuration $X$ is a $k$-antipodal point configuration in $\Rd$ if and only if there exists $\K_{\oo}\in {\cal{K}}_{\oo}^{d}$ such that $X$ is a $k$-diametral point configuration in $\Md_{\K_{\oo}}$, where $d\geq 2$ and $k\geq 2$. Moreover,
\begin{equation}\label{Main-1}
F_k(d)=\max_{\K_{\oo}\in {\cal{K}}_{\oo}^{d}} f_k(\Md_{\K_{\oo}})=(k-1)2^d
\end{equation}
holds for all $d\geq 2$ and $k\geq 2$. Furthermore, $f_k(\Md_{\K_{\oo}})=(k-1)2^d$ if and only if $\K_{\oo}$ is an $\oo$-symmetric affine $d$-cube of $\Rd$ in which case every $k$-diametral point configuration of $(k-1)2^d$ points is identical to the vertex set of a homothetic affine $d$-cube with each vertex having multiplicity $k-1$.
\end{theorem}

From the point of view of geometry, it is natural to complete this section with $k$-diametral (resp., $k$-antipodal) properties of {\it point sets}, i.e., point configurations consisting of distinct points. By restricting Definition \ref{k-diametral} (resp., Definition \ref{k-antipodal}) to point sets, let us denote the largest cardinality of $k$-diametral (resp., $k$-antipodal) point sets in $\Md_{\K_{\oo}}$ (resp., $\Rd$), by $g_k(\Md_{\K_{\oo}})$ (resp., $G_k(d)$) and call it the {\rm $k$-diametral number of point sets} (resp., {\rm $k$-antipodal number of point sets}) in $\Md_{\K_{\oo}}$ (resp., $\Rd$). Clearly, $g_k(\Md_{\K_{\oo}})\leq f_k(\Md_{\K_{\oo}})$ holds for all $k\geq 2$, $d\geq 2$ and $\K_{\oo}\in {\cal{K}}_{\oo}^{d}$ and so, the upper bounds already stated for $f_k(\Md_{\K_{\oo}})$ apply to $g_k(\Md_{\K_{\oo}})$ as well. Also, it is obvious that $G_k(d)\leq F_k(d)=(k-1)2^d$ holds for all $k\geq 3$, $d\geq 2$ with $G_2(d)=F_2(d)=2^d$ for all $d\geq 2$. Furthermore, we have

\begin{theorem}\label{Euclidean-k-diametral-point-sets}
\item{\bf (i)} $G_k(2)=2k$ for all $k\geq 2$. Furthermore, $S \subset \Rr^2$ is a $k$-antipodal point set with ${\rm card}(S)=2k$, if and only if $\PP:={\rm conv} (S)$ is a $(2s)$-gon for some $s \leq k$ with $S \subset {\rm bd} (\PP)$ such that each side of $\PP$ is parallel to another side of $\PP$ with both of them containing the same number of points from $S$.

\item{\bf (ii)} The point set $X$ is a $k$-antipodal point set in $\Rd$ if and only if there exists $\K_{\oo}\in {\cal{K}}_{\oo}^{d}$ such that $X$ is a $k$-diametral point set in $\Md_{\K_{\oo}}$ implying that $G_k(d)=\max_{\K_{\oo}\in {\cal{K}}_{\oo}^{d}} g_k(\Md_{\K_{\oo}})$,
where $d\geq 2$ and $k\geq 2$. Moreover, for all $d\geq 3$ and $k\geq 3$ one has
\begin{equation}\label{Main-2}
k\cdot 2^{d-1}\leq G_k(d) \leq (k-1)2^d-1 .
\end{equation}
\item{\bf (iii)} If $\mathbf{C}_{\oo}$ is an $\oo$-symmetric affine $d$-cube of $\Rd$, then $g_k(\Md_{\mathbf{C}_{\oo}})=k\cdot 2^{d-1}$ holds for all $d\geq 2$ and $k\geq 2$.
\item{\bf (iv)} For any $k$-diametral point set $S \subset \Ee^2$, we have ${\rm card}(S) \leq 2k-1$, with equality if and only if $S$ is the vertex set of a regular $(2k-1)$-gon. Thus, $g_k(\Ee^2)=2k-1$ for all $k\geq 2$.
\item{\bf (v)} $2k\leq g_k(\Ee^3)\leq 3k-2$ holds for all $k\geq 4$. Furthermore, $g_3(\Ee^3)=6$, and if $S \subset \Ee^3$ is a $3$-diametral point set with ${\rm card}(S)=6$, then the diameter graph of $S$ is isomorphic to one of (1-a), ..., (2-e) in Figure~\ref{fig:graphs}.
\end{theorem}

\begin{figure}[ht]
\begin{center}
\includegraphics[width=0.7\textwidth]{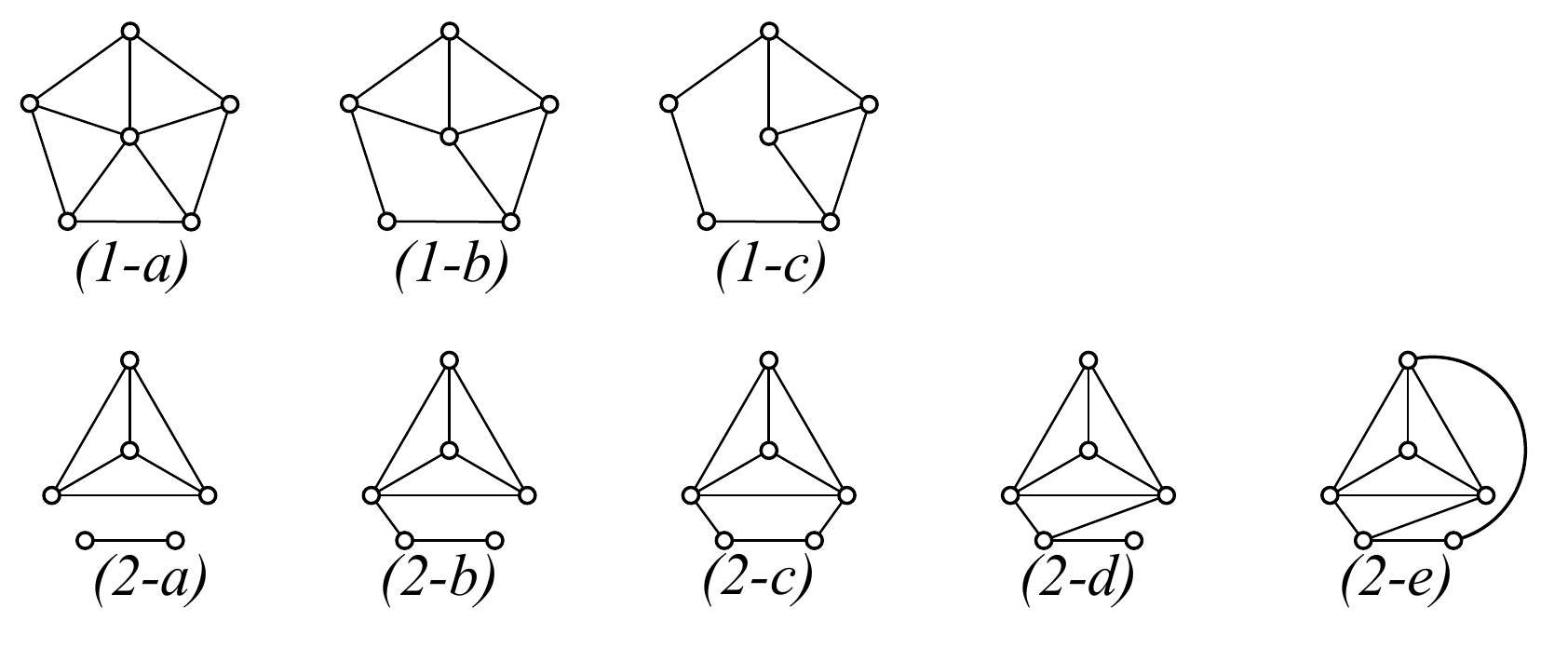}
\caption{The diameter graphs of $3$-diametral point sets in $\Ee^3$ of maximal cardinality.}
\label{fig:graphs}
\end{center}
\end{figure}

\begin{remark}\label{Polyanskii-conjecture}
We note that Conjecture 16 of \cite{Po}) states that $g_3(\Ed)\leq\big\lfloor\frac{3(d+1)}{2}\big\rfloor$. Part {\bf (iv)} (resp., Part {\bf (v)}) of Theorem \ref{Euclidean-k-diametral-point-sets} disproves (resp., proves) this conjecture for $d=2$ (resp., $d=3$).
\end{remark}

Part {\bf (ii)} of Theorem \ref{Euclidean-k-diametral-point-sets} supports
\begin{problem}\label{max-k-antipodal-point-sets}
Prove or disprove that $G_k(d)= k\cdot 2^{d-1}$ holds for all $d\geq 3$ and $k\geq 3$.
\end{problem}

In the rest of the paper we prove the theorems stated. In the proofs convexity methods are combined with volumetric estimates and combinatorial properties of diameter graphs.

\section{Proof of Proposition \ref{basic-estimates-Part1}}

{\bf Proof of Part (i)}. As any equilateral point configuration of $f_2(\Md_{\K_{\oo}})$ points consists of distinct points in $\Md_{\K_{\oo}}$, therefore by assigning multiplicity $k-1$ to each of those points one obtains a $k$-diametral point configuration in $\Md_{\K_{\oo}}$. Thus,
$(k-1)f_2(\Md_{\K_{\oo}})\leq f_k(\Md_{\K_{\oo}})$ holds for all $d\geq 2$, $k\geq 2$, and $\K_{\oo}\in {\cal{K}}_{\oo}^{d}$. Next, let $X:=\{\x_1,\x_2, \dots , \x_{f_k(\Md_{\K_{\oo}})}\}\subset\Rd$ be a $k$-diametral point configuration of $f_k(\Md_{\K_{\oo}})$ points in $\Md_{\K_{\oo}}$. One can think of $X$ as a point set of $\Rd$ consisting of some number of distinct points with each point having some multiplicity and then partition that underlying point set into $b(\Md_{\K_{\oo}})$ subsets say, $X_1, X_2, \dots , X_{b(\Md_{\K_{\oo}})}$ such that ${\rm diam}_{\Md_{\K_{\oo}}}(X_i)<{\rm diam}_{\Md_{\K_{\oo}}}(X)$ for all $1\leq i\leq b(\Md_{\K_{\oo}})$. As $X:=\{\x_1,\x_2, \dots , \x_{f_k(\Md_{\K_{\oo}})}\}\subset\Rd$ is a $k$-diametral point configuration in $\Md_{\K_{\oo}}$ therefore ${\rm card} \left(\{ j\ |\ \x_j\in X_i\}\right)\leq k-1$ holds for all $1\leq i\leq b(\Md_{\K_{\oo}})$, where ${\rm card}(\cdot)$ denotes the cardinality of the corresponding set. Thus, $f_k(\Md_{\K_{\oo}})=\sum_{i=1}^{b(\Md_{\K_{\oo}})} {\rm card} \left(\{ j\ |\ \x_j\in X_i\}\right)\leq (k-1) b(\Md_{\K_{\oo}})$ holds for all $d\geq 2$, $k\geq 2$, and $\K_{\oo}\in {\cal{K}}_{\oo}^{d}$. Finally, based on the results of \cite{Ro} (\cite{RoZo}),  \cite{Las}, and \cite{HSTV}, which have been discussed in the Introduction, we are left to show that $b(\Md_{\K_{\oo}})\leq h(\Rd)$ holds for all $d\geq 2$ and $\K_{\oo}\in {\cal{K}}_{\oo}^{d}$. This follows from the simple observation that if $\emptyset\neq Y\subset \Rd$ is a finite set and $0<\lambda<1$, $\y\in \Rd$, then ${\rm diam}_{\Md_{\K_{\oo}}}(\y+\lambda{\rm conv}(Y))<{\rm diam}_{\Md_{\K_{\oo}}}\left({\rm conv}(Y)\right)={\rm diam}_{\Md_{\K_{\oo}}}(Y)$ for all $d\geq 2$ and $\K_{\oo}\in {\cal{K}}_{\oo}^{d}$.

\noindent{\bf Proof of Part (ii)}. Clearly, the inequalities $f_2(\Mm^2_{\K_{\oo}})\leq \frac{1}{k-1}f_k(\Mm^2_{\K_{\oo}})\leq b(\Mm^2_{\K_{\oo}})$ (see Part {\bf (i)} of Proposition \ref{basic-estimates-Part1}) combined with the following claim complete the proof of Part {\bf (ii)}.

\begin{sublemma}\label{planar}
$f_2(\Mm^2_{\K_{\oo}})=b(\Mm^2_{\K_{\oo}})$ holds for all $\K_{\oo}\in {\cal{K}}_{\oo}^{2}$.
\end{sublemma}

\begin{proof}
On the one hand, recall that according to \cite{BoSo} and \cite{Gru} $b(\Mm^2_{\K_{\oo}})=3$ if $\K_{\oo}\in {\cal{K}}_{\oo}^{2}$ is different from a parallelogram and  $b(\Mm^2_{\K_{\oo}})=4$ if $\K_{\oo}$ is a parallelogram. On the other hand, Petty \cite{Pe}
proved that $f_2(\Mm^2_{\K_{\oo}})=3$ if $\K_{\oo}\in {\cal{K}}_{\oo}^{2}$ is different from a parallelogram and  $f_2(\Mm^2_{\K_{\oo}})=4$ if $\K_{\oo}$ is a parallelogram. Thus, Sublemma \ref{planar} follows.
\end{proof}

\noindent{\bf Proof of Part (iii)}. Let us start by recalling that $b(\Ee^2)=3$ (\cite{Bo}) and $b(\Ee^3)=4$ (\cite{Eg}, \cite{Grun}, and \cite{He}). Combining these facts with $f_2(\Ed)\leq \frac{1}{k-1}f_k(\Ed)\leq b(\Ed)$ (see Part {\bf (i)} of Proposition \ref{basic-estimates-Part1}) and the elementary observations that $f_2(\Ee^2)=3$ and $f_2(\Ee^3)=4$, one obtains that $f_k(\Ed)=(k-1)(d+1)$ holds for all $k\geq 2$ and $d=2,3$. In order to prove (\ref{Euclidean-bounds-1}), we observe that the inequalities $(k-1)(d+1)\leq f_k(\Ed)\leq (k-1)b(\Ed)$ follow from Part {\bf (i)} of Proposition \ref{basic-estimates-Part1} and the elementary observation that $f_2(\Ed)=d+1$. Finally, (\ref{Euclidean-bounds-2}) follows from (\ref{Euclidean-bounds-1}) and the results of Lassak \cite{La} and Schramm \cite{Sch} in a straightforward way. This completes the proof of Proposition \ref{basic-estimates-Part1}.

\section{Proof of Theorem \ref{basic-estimates-Part2}}

Let $X:=\{\x_1,\x_2, \dots , \x_n\}\subset\Rd$ be a $k$-diametral point configuration of $n$ points in $\Md_{\K_{\oo}}$ with diameter $t:={\rm diam}_{\Md_{\K_{\oo}}}(X)>0$. In what follows it will be convenient to use the notation $B^d_{\K_{\oo}}
[\y , r]$ for the (closed) ball of center $\y$ and radius $r>0$ in $\Md_{\K_{\oo}}$. We claim that $X$ is a $k$-antipodal point configuration of $n$ points in $\Rd$. In order to prove this claim, let $\x_{n_1}, \x_{n_2}, \dots , \x_{n_k}$, 
$1\leq n_1< n_2 < \dots  < n_k\leq n$ be a $k$-tuple chosen from $X$. By assumption there are two points say, $\x_{n_i}$ and $\x_{n_j}$, $1\leq i<j\leq k$ such that $\| \x_{n_i}- \x_{n_j}\|_{\K_{\oo}} =t$. Thus, 
\begin{equation}\label{equivalence-1}
{\rm conv}(X)\subseteq B^d_{\K_{\oo}}[\x_{n_i} , t] \cap B^d_{\K_{\oo}}[\x_{n_j} , t] \ {\rm with}\ x_{n_j}\in {\rm bd}(B^d_{\K_{\oo}}[\x_{n_i} , t]) \ {\rm and}\ x_{n_i}\in {\rm bd}(B^d_{\K_{\oo}}[\x_{n_j} , t]).
\end{equation}
Clearly, (\ref{equivalence-1}) implies the existence of a supporting hyperplane $H_i$ (resp., $H_j$) of $B^d_{\K_{\oo}}[\x_{n_j} , t]$ (resp., $B^d_{\K_{\oo}}[\x_{n_i} , t]$) with $\x_{n_i}\in H_i$ (resp., $\x_{n_j}\in H_j$) such that $H_i$ and $H_j$ are two distinct parallel hyperplanes between which ${\rm conv}(X)$ lies in $\Rd$. Hence, $X$ is a $k$-antipodal point configuration of $n$ points in $\Rd$. Next, we need to prove the converse of this statement. So, let $X:=\{\x_1,\x_2, \dots , \x_n\}\subset\Rd$ be a $k$-antipodal point configuration of $n$ points in $\Rd$. Without loss of generality we may assume that $\mathbf{P}_{\oo}:={\rm conv}(X)-{\rm conv}(X)\in {\cal{K}}_{\oo}^{d}$. We claim that $X:=\{\x_1,\x_2, \dots , \x_n\}\subset\Rd$ is a $k$-diametral point configuration of $n$ points in $\Md_{\mathbf{P}_{\oo}}$. In order to see this let $\x_{n_1}, \x_{n_2}, \dots , \x_{n_k}$, $1\leq n_1< n_2 < \dots  < n_k\leq n$ be a $k$-tuple chosen from $X$. By assumption there exist two points say, $\x_{n_i}$ and $\x_{n_j}$, $1\leq i<j\leq k$ lying on distinct parallel supporting hyperplanes of the convex hull ${\rm conv} (X)$ of $X$ in $\Rd$. Thus,
\begin{equation}\label{equivalence-2}
{\rm conv}(X)\subseteq B^d_{\mathbf{P}_{\oo}}[\x_{n_i} , 1] \cap B^d_{\mathbf{P}_{\oo}}[\x_{n_j} , 1] \ {\rm with}\ x_{n_j}\in {\rm bd}(B^d_{\mathbf{P}_{\oo}}[\x_{n_i} , 1]) \ {\rm and}\ x_{n_i}\in {\rm bd}(B^d_{\mathbf{P}_{\oo}}[\x_{n_j} , 1]).
\end{equation}
Clearly, (\ref{equivalence-2}) implies that $X:=\{\x_1,\x_2, \dots , \x_n\}\subset\Rd$ is a $k$-diametral point configuration of $n$ points in $\Md_{\mathbf{P}_{\oo}}$ with diameter ${\rm diam}_{\Md_{\mathbf{P}_{\oo}}}(X)=1$. This completes the proof of the first statement in Theorem \ref{basic-estimates-Part2} implying also that
\begin{equation}\label{Main-1-part1}
F_k(d)=\max_{\K_{\oo}\in {\cal{K}}_{\oo}^{d}} f_k(\Md_{\K_{\oo}})
\end{equation}
holds for all $d\geq 2$ and $k\geq 2$. Hence, in order to complete the proof of (\ref{Main-1}) we need to show that
\begin{equation}\label{Main-1-part2}
F_k(d)=(k-1)2^d
\end{equation}
holds for all $d\geq 2$ and $k\geq 2$. By taking the vertices of an affine $d$-cube each with multiplicity $k-1$ in $\Rd$, one obtains that $(k-1)2^d\leq F_k(d)$. Thus, we are left to show that
\begin{equation}\label{Main-1-part3}
F_k(d)\leq (k-1)2^d
\end{equation}
holds for all $d\geq 2$ and $k\geq 2$. The following proof of (\ref{Main-1-part3}) is a natural extension of the volumetric method of \cite{DaGr}. The details are as follows. Let $X:=\{\x_1,\x_2, \dots , \x_n\}\subset\Rd$ be a $k$-antipodal point configuration of $n$ points in $\Rd$.
Without loss of generality we may assume that ${\rm int}(\mathbf{P})\neq\emptyset$, where $\mathbf{P}:={\rm conv}(X)\subset\Rd$. So, if ${\rm vol}_d(\cdot)$ denotes the $d$-dimensional Lebesgue measure in $\Rd$, then by assumption ${\rm vol}_d(\mathbf{P})>0$. Let
$\mathbf{P}_l:=\x_l+\frac{1}{2}\left(\mathbf{P}-\x_l\right)$ for $1\leq l\leq n$. Clearly, $\mathbf{P}_l\subset\mathbf{P}$ for all $1\leq l\leq n$. We claim that the convex $d$-polytopes $\mathbf{P}_1, \mathbf{P}_2, \dots , \mathbf{P}_n$ form a $(k-1)$-fold packing in the convex $d$-polytope $\mathbf{P}$, i.e., for any $k$-tuple
$\mathbf{P}_{n_1}, \mathbf{P}_{n_2}, \dots , \mathbf{P}_{n_k}$, $1\leq n_1< n_2 < \dots  < n_k\leq n$ one has
\begin{equation}\label{packing}
{\rm int}\left(\mathbf{P}_{n_1}\right)\cap {\rm int}\left(\mathbf{P}_{n_2}\right)\cap \dots \cap{\rm int}\left(\mathbf{P}_{n_k}\right)=\emptyset.
\end{equation} 
Namely, by assumption the $k$-tuple $\x_{n_1}, \x_{n_2}, \dots , \x_{n_k}$, $1\leq n_1< n_2 < \dots  < n_k\leq n$ chosen from $X$ must contain two points say, $\x_{n_i}$ and $\x_{n_j}$, $1\leq i<j\leq k$ lying on distinct parallel supporting hyperplanes of $\mathbf{P}$ in $\Rd$, i.e., there must exist a $(d-1)$-dimensional linear subspace $H\subset \Rd$ such that $\x_{n_i}+H$ and $\x_{n_j}+H$, $1\leq i<j\leq k$ are distinct parallel supporting hyperplanes of $\mathbf{P}$. This implies in a straightforward way that $\frac{1}{2}(\x_{n_i}+\x_{n_j})+H$ separates $\mathbf{P}_{n_i}$ and $\mathbf{P}_{n_j}$, finishing the proof of (\ref{packing}). Finally, as $\mathbf{P}_1, \mathbf{P}_2, \dots , \mathbf{P}_n$ form a $(k-1)$-fold packing in $\mathbf{P}$ therefore
\begin{equation}\label{Main -1-part4}
\frac{n}{2^d}{\rm vol}_d(\mathbf{P})=\sum_{l=1}^n{\rm vol}_d(\mathbf{P}_l)\leq (k-1){\rm vol}_d(\mathbf{P}),
\end{equation}
from which (\ref{Main-1-part3}) follows, completing the proof of (\ref{Main-1-part2}).

We are left to characterize the case of equality in (\ref{Main-1-part3}). The above proof of (\ref{Main-1-part3}) shows that equality in (\ref{Main-1-part3}) implies the existence of a convex $d$-polytope $\mathbf{P}:={\rm conv}(X)\subset\Rd$ and a point configuration $X:=\{\x_1,\x_2, \dots , \x_N\}$ of $N=(k-1)2^d$ points in $\Rd$ such that the convex d-polytopes $\mathbf{P}_l:=\x_l+\frac{1}{2}\left(\mathbf{P}-\x_l\right)$, $1\leq l\leq N$, which are translates of each other and are half-size homothetic copies of $\mathbf{P}$ lying in $\mathbf{P}$, form a {\it $(k-1)$-tiling} in $\mathbf{P}$, i.e., they possess the property given in the following definition.

\begin{definition}\label{multiple-tiling}
We say that the convex $d$-polytopes $\mathbf{P}_l\subset \Rd$, $1\leq l\leq N$ form a {\rm $(k-1)$-tiling} of the convex $d$-polytope $\mathbf{P}\subset\Rd$ if $\mathbf{P}_l\subset\mathbf{P}$ holds for all $1\leq l\leq N$ and every point of $\mathbf{P}$ which is not a boundary point of any  $\mathbf{P}_l, 1\leq l\leq N$ belongs to the interior of exactly $k-1$ convex $d$-polytopes chosen from $\mathbf{P}_l, 1\leq l\leq N$. 
\end{definition}

We need to prove

\begin{lemma}\label{characterization}
Let $\mathbf{P}:={\rm conv}(X)\subset\Rd$, $d\geq 1$ be a convex $d$-polytope and $X:=\{\x_1,\x_2, \dots , \x_N\}$ be a point configuration of $N=(k-1)2^d$ points in $\Rd$ with $k\geq 2$ such that the convex d-polytopes $\mathbf{P}_l:=\x_l+\frac{1}{2}\left(\mathbf{P}-\x_l\right)$, $1\leq l\leq N$, form a $(k-1)$-tiling in $\mathbf{P}$. Then $\mathbf{P}$ is an affine $d$-cube and $X$ is its vertex set with each vertex having multiplicity $k-1$.
\end{lemma}

\begin{proof}
We prove Lemma \ref{characterization} by induction on $d$. We leave the easy proof of Lemma \ref{characterization} for $d=1$ and for all $k\geq 2$ to the reader and assume that it holds in dimensions at most $d-1\geq 1$ and for all $k\geq 2$. Next, let $\mathbf{P}$ and $X$ be given with the properties assumed in Lemma \ref{characterization}. Let $F$ be an arbitrary facet, i.e., $(d-1)$-dimensional face of $\mathbf{P}$ and let $F_{\epsilon}:=\cup \{B^d(\x, \epsilon)| \x\in F\}$ be the $\epsilon$-neighbourhood of $F$ in $\Ed$, where $B^d(\x, \epsilon)$ denotes the open ball of radius $\epsilon$ centered at $\x$ in $\Ed$. Clearly, there exists $\epsilon_{0}>0$ such that $F_{\epsilon_{0}}\cap \mathbf{P}_l\neq\emptyset$ if and only if $F \cap \mathbf{P}_l\neq\emptyset$, where $1\leq l\leq N$. Let $I_F:=\{ l \ |\ F \cap \mathbf{P}_l\neq\emptyset ,1\leq l\leq N\}$ and $X_F:=\{\x_l\ |\ l\in I_F\}$. 

\begin{sublemma}\label{local-version}
There exists an affine $d$-cube $\mathbf{C}_F\subset\Rd$ such that $F$ is a facet of $\mathbf{C}_F$, $F_{\epsilon_{0}}\cap \mathbf{C}_F = F_{\epsilon_{0}}\cap \mathbf{P}$, and $X_F$ is the vertex set of $F$ with each vertex having multiplicity $k-1$.
\end{sublemma}
\begin{proof}
Clearly, $F={\rm conv} (X_{F})$. Furthermore, as the convex $d$-polytopes $\mathbf{P}_l$, $1\leq l\leq N$, form a $(k-1)$-tiling in $\mathbf{P}$ therefore the convex $(d-1)$-polytopes  $F\cap\mathbf{P}_l=\x_l+\frac{1}{2}\left(F-\x_l\right)$, $l\in I_{F}$ must form  a $(k-1)$-tiling in $F$. By induction it follows that $F$ is an affine $(d-1)$-cube and $X_{F}$ is its vertex set with each vertex having multiplicity $k-1$. Next, let $F'$ and $F''$ be an arbitrary pair of facets of $ \mathbf{P}$ having the property that $F'\cap F$ and $F''\cap F$ are distinct parallel $(d-2)$dimensional faces of $F$. Then choose $l',l''\in I_F$ such that $F\cap\mathbf{P}_{l'}=\x_{l'}+\frac{1}{2}\left(F-\x_{l'}\right)$ and $F\cap\mathbf{P}_{l''}=\x_{l''}+\frac{1}{2}\left(F-\x_{l''}\right)$ share a $(d-2)$-dimensional face in common which is parallel to $F'$ and $F''$. Using the $(k-1)$-tiling of $\mathbf{P}$ again, it follows that $\mathbf{P}_{l'}$ and $\mathbf{P}_{l''}$ must share a facet in common implying that $F'$ and $F''$ are parallel facets of $ \mathbf{P}$. The existence of $\mathbf{C}_F$ follows, finishing the proof of Sublemma \ref{local-version}.
\end{proof}
It is easy to see that Sublemma \ref{local-version} applied to the facets of $\mathbf{P}$ finishes the proof of Lemma \ref{characterization}.
\end{proof}
Finally, let $X:=\{\x_1,\x_2, \dots , \x_N\}$ be a $k$-diametral point configuration of $N=(k-1)2^d$ points in $\Md_{\K_{\oo}}$. By the first part of Theorem \ref{basic-estimates-Part2} it follows that $X$ is a $k$-antipodal point configuration of $N=(k-1)2^d$ points in $\Rd$. Then using Lemma \ref{characterization}, one obtains that ${\rm conv}(X)$ is an affine $d$-cube in $\Rd$ and $X$ is its vertex set with each vertex having multiplicity $k-1$. Thus, $\K_{\oo}$ must be an $\oo$-symmetric affine $d$-cube homothetic to ${\rm conv}(X)$ in $\Rd$, which completes the proof of Theorem \ref{basic-estimates-Part2}.

\section{Proof of Theorem \ref{Euclidean-k-diametral-point-sets}}

{\bf Proof of Part (i)}. We prove the statement by induction on $k$. For $k=2$ it is trivial, and hence, we assume that for some $k \geq 3$ we have $G_{k-1}(2)=2k-2$, and any $(k-1)$-antipodal point set of cardinality $2k-2$ satisfies the conditions of the theorem. Let $S$ be a $k$-antipodal point set of maximal cardinality in $\Rr^2$, and let $\PP := {\rm conv} (S)$.
Then, clearly, $\mathrm{card} (S) \geq 2k$, ${\rm int} (\PP)\neq\emptyset$, and $S \subset {\rm bd} (\PP)$, as $\p \in \inter (\PP)$ would imply that $S \setminus \{ \p \}$ is a $(k-1)$-antipodal point set of cardinality $2k-1 > 2k-2$.
Let the points of $S$ be $\p_1, \p_2,\ldots, \p_m$ in counterclockwise order in ${\rm bd}(\PP)$. Observe that any $\p_i$ is antipodal to some consecutive vertices of $\PP$. Furthermore, if $\p_i$ is antipodal to $\p_{j(i)}, \p_{j(i)+1}, \ldots, \p_{t(i)}$, then both $j(i)$ and $t(i)$ are increasing functions of $i$ with respect to the counterclockwise order.

Now, assume that some $\p_j$ is a relative interior point of a side of $\PP$; say, assume that for some value of $i \geq 3$, $\p_2, \p_3, \ldots, \p_{i-1}$ lie in the relative interior of the side $[\p_1,\p_i]$ of $\PP$. Assume that the supporting line of $\PP$ which is parallel to $[\p_1,\p_i]$ and is distinct from the line of $[\p_1,\p_i]$, contains the points $\p_s,\p_{s+1},\ldots, \p_{t}$ for some $i < s \leq t \leq m$. Note that after suitably relabeling the points we can achieve that $i-1 \geq t-s$, and $s-i \geq m+1-t$, implying $s \geq \frac{m}{2}+1$.
Since $S$ is $k$-antipodal therefore the points $\p_2,\p_3,\ldots,\p_{k+1}$ must contain an antipodal pair, which is equivalent to the inequality $k+1 \geq s$.
Hence, $k+1 \geq \frac{m}{2}+1$ and so, $m\leq 2k$. 
For later use we remark that if $m=2k$, then the $k$-antipodality of $S$ yields that everywhere in the above chain of inequalities we have equality. In particular, in this case we have $s = k+1$ and $i-1=t-s$, implying also that $t=k+i$.

To show that $G_k(2) = 2k$, we are left with the case that $S$ is the vertex set of $\PP$.
We show that some $\p_i$ is antipodal to at most three vertices of $\PP$. Indeed, if $\p_i$ is antipodal to $\p_{j},\p_{j+1},\p_{j+2}, \p_{j+3}$, then $\p_{j+1}$ and $\p_{j+2}$ are antipodal to at most two vertices of $\PP$.
Thus, without loss of generality, we may assume that $t(1) \leq j(1)+2$. Then, using a suitable labeling of the vertices, we may also assume that
$j(1) \geq \frac{m}{2}$ if $m$ is even, and $j(1) \geq \frac{m+1}{2}$ if $m$ is odd. Thus, since among $\p_1,\p_2,\ldots, \p_k$ there is an antipodal pair, we have $m  \leq 2k$ if $m$ is even (since otherwise $t(k) \leq m$), and $m \leq 2k-1$ if $m$ is odd (for a similar reason). In particular, the vertices of a regular $(2k)$-gon are $k$-antipodal. Furthermore, observe that if $t(1) \leq j(1)+1$, then in the same way it follows that $m \leq 2k-1$.

Finally, assume that ${\rm card}(S)=2k$ and $S$ is $k$-antipodal.
If $S$ is the vertex set of $\PP:= {\rm conv}(S)$, then ${\rm card}(S)=2k$ and the argument in the previous paragraph yields that every vertex is antipodal to exactly three vertices of $P$. In addition, it follows similarly that $j(i) = k+i$ for all values of $i$.
Thus, for all values of $i$ the lines through $\p_{i+k}$ and parallel to $[\p_{i-1},\p_i]$ or $[\p_i,\p_{i+1}]$ are supporting lines of $P$, and similarly, the lines through $\p_i$ and parallel to $[\p_{k+i-1},\p_{k+i}]$ or $[\p_{k+i},\p_{k+i+1}]$ are supporting lines of $P$. Hence, $[\p_i,\p_{i+1}]$ and $[\p_{i+k},\p_{i+k+1}]$ are parallel to all values of $i$, yielding that $P$ is a $(2k)$-gon with the property that its opposite pairs of sides are parallel.
Finally, consider the case that the relative interior of a side of $\PP$ contains a point of $S$ say, $\p_2, \p_3, \ldots, \p_{i-1}$ lie in the relative interior of the side $[\p_1,\p_i]$ of $\PP$. Using the notation and the argument in the first two paragraphs of our proof of Part \textbf{(i)}, we have that $\p_{k+1},\p_{k+2},\ldots,\p_{k+i}$ lie on the same side $[\p_{k+1},\p_{k+i}]$, and this side is parallel to $[\p_1,\p_{i}]$. On the other hand, removing all points $\p_2,\ldots,\p_{i-1}, \p_{k+2},\ldots,\p_{k+i-1}$, we obtain a $(k-i+2)$-antipodal set $S'$ of $2(k-i+2)$ points with $\PP = \mathrm{conv} (S')$. Then, by the induction hypothesis, $P$ is an even-sided polygon which contains the same number of points of $S'$ on any pair of opposite sides, implying that the points of $S$ satisfy the same property. {This completes the proof of Part {\bf (i)}.

{\bf Proof of Part (ii)}. The first claim and the upper estimate of (\ref{Main-2}) follow from the proof of Theorem \ref{basic-estimates-Part2} in a straightforward way. So, we are left to prove  the lower bound of (\ref{Main-2}). Let us start with the $(2s)$-gon $\PP:={\rm conv} (S)\subset \Ee^2$ for some $s \leq k$ and ${\rm card}(S)=2k$ having the property that $S \subset {\rm bd} (\PP)$ and each side of $\PP$ is parallel to another side of $\PP$ with both of them containing the same number of points from $S$. Then let $\Ee^2\subset\Ee^3$ be a plane through the origin of $\Ee^3$ and let $\PP':=\x +\PP$ for some $\x\in\Ee^3\setminus\Ee^2$. Clearly, ${\rm conv} (\PP\cup\PP')$ is a $1$-fold prism in $\Ee^3$ and Part {\bf (i)} implies that $S\cup (\x+S)$ is a $k$-antipodal point set of cardinality $2(2k)$ in $\Ee^3$ lying on the edges of ${\rm conv} (\PP\cup\PP')$. Repeating this process $(d-2)$-times one obtains  a $(d-2)$-fold prism in $\Ed$ such that its $1$-skeleton contains a $k$-antipodal point set of cardinality $2^{d-1}\cdot k$ in $\Ed$, finishing the proof of Part {\bf (ii)}.


{\bf Proof of Part (iii)}. Let $S$ be a $k$-diametral point set in $\Md_{\mathbf{C}_{\oo}}$. Then, without loss of generality one may assume that $S \subset {\rm bd}(\mathbf{C})$, where $\mathbf{C}$ denotes the smallest rectangular box containing $S$ and having facets parallel to the corresponding facets of $\mathbf{C}_{\oo}$. Let $V(\mathbf{C})$ be the vertex set of $\mathbf{C}$, and let $v$ denote the number of points in $S \cap V(\mathbf{C})$. Finally, for any $\p \in V(\mathbf{C})$, let $F_{\p}$ denote the union of the facets of $\mathbf{C}$ that do not contain $\p$, and set $w:=\mathrm{card} (S) - v$. Observe that by $k$-diametrality, we have that $\mathrm{card}(S \setminus F_{\p} ) \leq k-1$ for every vertex $\p \in V(\mathbf{C})$, and that if $\q \in S \setminus V(\mathbf{C})$, then $\q \in {\rm bd}(\mathbf{C}) \setminus F_{\p}$ for at least two distinct points $\p \in V(\mathbf{C})$. Thus,
\[
2\ {\rm card}(S)-v=v+2w \leq \sum_{\p \in V(\mathbf{C})} \mathrm{card} (S \setminus F_{\p}) \leq (k-1) V(\mathbf{C}) = (k-1)2^d.
\]
Combining this inequality with the trivial inequality $v \leq 2^d$, it readily follows that $\mathrm{card}(S) \leq k \cdot 2^{d-1}$. This together with the relevant construction of the proof of Part {\bf (ii)} completes the proof of Part {\bf (iii)}.

{\bf Proof of Part (iv)}. 
We prove the statement by induction on $k$. Clearly, $g_2(\Ee^2)=3$, and it is attained only for the vertex set of a regular triangle. Now, assume that the statement holds for any $l$-diametral point set in $\Ee^2$ for all $2\leq l\leq k-1$ and let $G$ be the diameter graph of the $k$-diametral point set $S \subset \Ee^2$. Here the vertex set of $G$ is $S$, and two vertices are connected by an edge if and only if their distance is equal to the diameter of $S$.

If a component of $G$ is a singleton $\{\p \}$, then we can apply the induction hypothesis to $S \setminus \{ \p \}$.

Next, assume that $G$ contains a $1$-valent vertex $\p$. Let the neighbor of $\p$ be denoted by $\q$.
Then any subset of cardinality $k-1$ in $S \setminus \{ \p,\q\}$ contains a diameter of $S$, as adding $\p$ to it contains a diameter of $S$, and this diameter does not contain $\p$.
Thus, in this case ${\rm card}(S \setminus \{ \p,\q \}) \leq 2k-3$, with equality if and only if it is the vertex set of a regular $(2k-3)$-gon. If  ${\rm card}(S \setminus \{ \p,\q \}) < 2k-3$, then we are done by induction. Otherwise
$S$ is contained in the Reuleaux-polygon defined by the vertices of a regular $(2k-3)$-gon, which yields that $\p,\q$ can be connected only to vertices of this polygon, and not to each other, a contradiction.

We are left with the case that every vertex of $G$ is connected to at least two other vertices. In this case $G$ contains an $m$-cycle $C$ for some value of $m$.
Observe that then there is a Reuleaux-polygon whose vertices are the vertices of $C$; this Reuleaux-polygon is obtained as the intersection of $m$ congruent disks centered at the vertices of $C$, and with radius equal to the diameter of $S$.
This Reuleaux-polygon contains $S$ in its boundary, which yields that $G$ is connected, and it is a cycle with some additional $1$-valent vertices attached to some of its vertices. Since we assumed that every vertex of $G$ has degree at least $2$, it follows that $G=C$, implying that 
${\rm card}(S) = m$. On the other hand, if $m \geq 2k$, then we can choose $k$ mutually disconnected vertices of $S$, which contradicts our assumption that $S$ is $k$-diametral. Thus, $m \leq 2k-1$, finishing the proof of Part {\bf (iv)}.

{\bf Proof of Part (v)}. First, we assume that $k \geq 4$. The estimate $2k\leq g_k(\Ee^3)$ is obtained from the example of the vertex set of a regular $(2k-1)$-gon $P\subset\Ee^3$, and an additional point in $\Ee^3$ whose distance from all vertices of $P$ is equal to the diameter of $P$.

Now, we prove that $g_k(\Ee^3)\leq 3k-2$. Consider a $k$-diametral point set $S \subset \Ee^3$. Let $G$ be the diameter graph of $S$. If $G$ does not contain an odd cycle then it is bipartite, and by $k$-diametrality we have $\mathrm{card} (S) \leq 2k-2$. Thus, we may assume that $G$ contains an odd cycle. Let $C$ be a shortest odd cycle of $G$ and let $G \setminus C$ denote the graph obtained by removing the vertices of $C$ from $G$, and also all edges of $G$ containing any vertex of $C$.
By \cite{Do}, we have that any two odd cycles in $G$ intersect (cf. also \cite{Sw}). Thus, $G \setminus C$ contains no odd cycle, which implies that this graph is bipartite with partite sets say, $V_1$ and $V_2$ forming a partition of the vertices of $G \setminus C$ such that there is no edge of $G$ between any two vertices of $V_1$ (resp., $V_2$).

Let the vertices of $C$ be $\p_1,\p_2, \ldots,\p_{2s-1}$ in cyclic order. Observe that if an edge of $G$ connects $\p_i$ and $\p_j$ with $|i-j| \not\equiv 1 \mod (2s-1)$, then $G$ contains an odd cycle shorter than $C$; a contradiction. Thus, there is no such edge of $G$ in $C$. Furthermore, any two vertices $\p_i$ and $\p_j$ of $C$ divide $C$ into two paths, disjoint apart from their endpoints, exactly one of which is odd. Since $C$ is a shortest odd cycle in $G$, this implies that if some vertex in $V_1 \cup V_2$ is connected to both $\p_i$ and $\p_j$, then $|i-j| \equiv 2 \mod (2s-1)$. Thus, no vertex of $V_1 \cup V_2$ is connected to more than two vertices of $C$, and if some vertex in $V_1 \cup V_2$ is connected to $\p_i$ and $\p_j$ with $i \neq j$, then $|i-j| \equiv 2 \mod (2s-1)$. (For this idea see \cite{Do}, or equivalently \cite{HL}.) 
Let $s_i:={\rm card}(V_i)$ for $i=1,2$, and without loss of generality, we assume that $s_1 \geq s_2$. Clearly, since $S$ is $k$-diametral, we have $s \leq k$ and $s_1 \leq k-1$, and we have $n:=\mathrm{card}(S) \leq 2s+2s_1-1$.
We give an upper bound for the quantity $2s+2s_1-1$ in the following way.

Let $G'$ be a graph containing a $(2t-1)$-cycle $A$, with vertices $\p_1,\ldots,\p_{2t-1}$ in cyclic order, such that the graph $B:=G' \setminus A$ contains no edges and possesses $s_1$ vertices, and $G'$ does not contain an empty $k$-vertex graph as a subgraph. Furthermore, assume that any vertex of $B$ is connected to at most $2$ vertices of $A$, and if some vertex of $B$ is connected to some $\p_i$ and $\p_j$ with $i \neq j$, then $|i-j| \equiv 2 \mod (2s-1)$.
Clearly, any upper bound on ${\rm card}(A)+2{\rm card}(B)$ under these conditions is an upper bound for ${\rm card}(S)$.

Note that since $G'$ contains no empty $k$-vertex subgraph, we have $t \leq k$ and ${\rm card}(B) := t' \leq k-1$.
To upper bound $2t-1+2t'$, we may assume that any vertex of $B$ is connected to exactly two vertices of $A$, as otherwise vertices of $B$  with degree less than $2$ can be connected to other vertices of $A$ without violating our conditions. In other words, we assume that any vertex of $B$ is connected to $\p_{i-1}$ and $\p_{i+1}$ for some (unique) value of $i$. For any value of $i$, let $X_i$ denote the set of vertices in $B$ that are connected to $\p_{i-1}$ and $\p_{i+1}$. Clearly, the sets $X_1, X_2, \ldots, X_{2t-1}$ form a partition of $B$, and in particular, if $i \neq j$ then $X_i \cap X_j = \emptyset$. Set $w_i := 1+\mathrm{card} (X_i)$ for all values of $i$. Note that for any value of $i$, there is no edge in $G'$ between any pair of vertices in $\bigcup_{j=1}^{t-1} \left(  X_{i+2j} \cup \{ \p_{i+2j}\}\right)$. Since $G'$ contains no empty $k$-vertex subgraph, this implies that $\sum_{j=1}^{t-1} w_{i+2j} \leq k-1$ for all values of $i$.
Thus,
\[
(t-1) \sum_{i=1}^{2t-1} w_i = \sum_{i=1}^{2t-1} \left( \sum_{j=1}^{t-1} w_{i+2j} \right) \leq (2t-1)(k-1).
\]
On the other hand, $\sum_{i=1}^{2t-1} w_i = {\rm card}(A) + {\rm card}(B) = 2t-1+t'$.
Thus, $(t-1)(2t-1+t') \leq (2t-1)(k-1)$, implying that $t' \leq \left\lfloor \frac{(2t-1)(k-t)}{t-1} \right\rfloor$. This yields that $t' \leq \min \left\{ k-1, \left\lfloor \frac{(2t-1)(k-t)}{t-1} \right\rfloor \right\}$
under the condition that $2 \leq t \leq k$. From this it follows that
\[
2t+2t'-1 \leq \min \left\{ 2t+2k-3, \left\lfloor \frac{(2t-1)(2k-t-1)}{t-1} \right\rfloor \right\}.
\] 
An elementary computation shows that $2t+2k-3 \leq \left\lfloor \frac{(2t-1)(2k-t-1)}{t-1} \right\rfloor$ if and only if
$t \leq \left\lceil \frac{k}{2} \right\rceil$, and also that the expression on the right-hand side is decreasing on the interval $t \in [2,k]$ for all $k \geq 4$. Since for $2 \leq t \leq \left\lceil \frac{k}{2} \right\rceil$, we have $2t+2k-3 \leq 3k-2$ and for
$\left\lceil \frac{k}{2} \right\rceil + 1 \leq t \leq k$ we have $\left\lfloor \frac{(2t-1)(2k-t-1)}{t-1} \right\rfloor  \leq 3k-2$ for all $k \geq 4$, it follows that ${\rm card}(S) \leq 3k-2$, finishing the proof of $2k\leq g_k(\Ee^3)\leq 3k-2$ for all $k \geq 4$.

Now we characterize the diameter graphs of the $3$-diametral sets $S \subset \Ee^3$ with $\mathrm{card} (S) \geq 6$.
First we show that the graphs shown in Figure~\ref{fig:graphs} can be obtained as diameter graphs of such sets.
The wheel graph in (1-a) belongs, for instance, to the union of the vertex set of a regular pentagon and a point whose distance from all vertices is equal to the diameter of the pentagon. To construct (1-b) or (1-c) we can remove at most two consecutive `axles' of the graph by slightly moving one or two vertices of the pentagon towards the additional point.

To construct (2-a), consider the vertex set $V$ of a regular tetrahedron of unit edge length, and let $\mathbf{P}$ denote the intersection of the four closed unit balls centered at a point of $V$. This set is called a \emph{regular ball-tetrahedron}, with vertices $V$, and the edges of $\mathbf{P}$ are circle arcs in $\mathrm{bd}(\mathbf{P})$ (of radius $\frac{\sqrt{3}}{2}$) connecting two distinct vertices, and contained in the boundary of two of the balls generating $\mathbf{P}$ (cf. \cite{BLNP}).
An elementary computation shows that the distance of the midpoints of two opposite edges of $\mathbf{P}$ is $\sqrt{3}-\frac{\sqrt{2}}{2} \approx 1.02 > 1$, and hence, we may choose two points on this segment in the interior of $\mathbf{P}$ at unit distance. To construct the graphs in (2-b)-(2-e), we may move the endpoints of the segment chosen in the previous example in a suitable way.

In the remaining part we show that no other graph is the diameter graph of a $3$-diametral set $S \subset \Ee^3$ with $\mathrm{card} (S) \geq 6$.
To prove this, we need

\begin{lemma}\label{lem:diametral}
The edge graph of a quadrangle based pyramid (cf. Figure~\ref{fig:graphs2}) is not a subgraph of the diameter graph of any point set in $\Ee^3$.
\end{lemma}

\begin{figure}[ht]
\begin{center}
\includegraphics[width=0.18\textwidth]{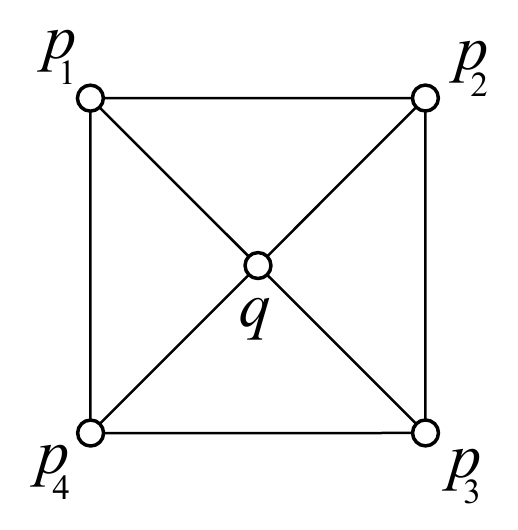}
\caption{A forbidden subgraph of diameter graphs of point sets in $\Ee^3$.}
\label{fig:graphs2}
\end{center}
\end{figure}

\begin{proof}
Suppose for contradiction that there  is some set $S=\{\p_1, \ldots, \p_4, \q \}$ of five distinct points in $\Ee^3$ such that the pairs connected by an edge in Figure~\ref{fig:graphs2} are diameters of $S$. Without loss of generality, we may assume that $\mathrm{diam}(S) = 1$ and that $\q$ is the origin, which yields that $\p_1, \ldots, \p_4$ are points of the sphere $\Ss^2$ such that the pairs $\p_i$, $\p_{i+1}$ are at spherical distance $\frac{\pi}{3}$ for $i=0,1,2,3$, where $\p_0 = \p_4$. From this it easily follows that these points are the vertices of a spherical rhombus of edge length $\frac{\pi}{3}$. On the other hand, from the triangle inequality it follows that in this case at least one diagonal is longer than $\frac{\pi}{3}$, which implies that $\mathrm{diam} (S) > 1$; a contradiction.
\end{proof}

Now we apply a similar argument as in the first part of the proof. Consider a $3$-diametral point set $S \subset \Ee^3$. Let $G$ be the diameter graph of $S$.
If $G$ is bipartite, then its vertex set $S$ can be partitioned into two parts such that no two vertices in the same part are connected. This, combined with the $3$-diametric property of $S$, clearly implies that $\mathrm{card} (S) \leq 4$. Thus, we restrict the investigation to graphs containing an odd cycle. 

Let $C$ be a shortest odd cycle of $G$ and let $G \setminus C$ denote the graph obtained by removing the vertices of $C$ from $G$, and also all edges of $G$ containing any vertex of $C$.
By \cite{Do}, we have that any two odd cycles in $G$ intersect. Thus, $G \setminus C$ contains no odd cycle, which implies that this graph is bipartite with partite sets say, $V_1$ and $V_2$ forming a partition of the vertices of $G \setminus C$ such that there is no edge of $G$ between any two vertices of $V_1$ (resp., $V_2$).

First, we show that $C$ is a $3$-cycle. By contradiction, assume that $C$ is not a $3$-cycle. Then the $3$-diametral property of $S$ implies that $C$ is a $5$-cycle.
Consider any vertex $\q$ of $G \setminus C$. Then for any nonadjacent pair of vertices of $C$, $\q$ is connected to at least one of them. This implies that $\q$ is connected to at least $3$ vertices of $C$, which yields that it is connected to two adjacent vertices. From this, it follows that $G$ contains a triangle; a contradiction, implying that $C$ is a $3$-cycle. Clearly, this implies that ${\rm card}(V_1), {\rm card}(V_2) \leq 2$, as $S$ is $3$-diametral. 
Let the vertices of $C$ be $\p_1, \p_2, \p_3$. We distinguish two cases.

\emph{Case 1}: $G$ does not contain $K_4$ (the complete graph on four vertices) as a subgraph.\\
First, we show that $\mathrm{card} (S) \leq 6$.
For a contradiction, assume that ${\rm card}(V_1)={\rm card}(V_2)=2$, and set $V_1 := \{ \x_1, \y_1 \}$ and $V_2 := \{ \x_2, \y_2\}$.
Since $G$ contains no $K_4$, any vertex of $V_1 \cup V_2$ is connected to at most two vertices $C$. Furthermore, since $S$ is $3$-diametral, no two vertices of $V_1 \cup V_2$ are connected to the same pair of vertices of $C$.

By $3$-diametrality, any vertex of $C$ is connected to at least one of $\x_1$ and $\y_1$.
This implies that $\x_1$ or $\y_1$ is connected to at least two vertices of $C$. Thus, without loss of generality, we may assume that $\x_1$ is connected to, say, $\p_1$ and $\p_2$, and not to $\p_3$, and $\y_1$ is connected to $\p_3$. Similarly, we may assume that $\x_2$ is connected to $\p_2,\p_3$ and not to $\p_1$, and $\y_2$ is connected to $\p_1$. Now Lemma~\ref{lem:diametral} implies that $\x_1$ and $\x_2$ are not connected. Furthermore, since $G$ is a diameter graph, by \cite{Do} $\x_1$ and $\y_2$ are not connected, as otherwise $\x_1, \y_2, \p_1$ and $\x_2, \p_2, \p_3$ would be two disjoint odd cycles. But then $\x_1, \x_2, \y_2$ are the vertices of an empty triangle, a contradiction.

Now we consider the case that $\mathrm{card}(S) = 6$. Then we may assume that $\mathrm{card} (V_1) = 2$ and $\mathrm{card} (V_2) = 1$, and set $V_1 := \{ \x_1, \y \}$, and $V_2 := \{ \x_2 \}$. Then, similarly like in the previous case, we may assume that $\x_1$ is connected to $\p_1$ and $\p_2$, and $\y$ is connected to $\p_3$.

\emph{Subcase 1.1}: $\y$ is not connected to $\p_1$ and $\p_2$.
Then, by $3$-diametrality, since the pairs $\{ \x_1, \y \}$, $\{ \x_1, \p_3 \}$, $\{ \y, \p_1 \}$ and $\{ \y, \p_2 \}$ are not connected, $\x_2$ is connected to at least one point from each pair. Thus, if $\x_2$ is not connected to $\y$, then it is connected to $\x, \p_1, \p_2$, and $G$ contains $K_4$; a contradiction, and we have that $\x_2$ is connected to $\y$. Furthermore, $\x_2$ is not connected to $\p_3$, as otherwise $G$ contains two disjoint triangles. This yields that $\x_2$ is connected to $\x_1$. Furthermore, since $G$ contains no $K_4$, $\x_2$ is not connected to $\p_1$ or $\p_2$. Since all other edges of $G$ have been determined, we have that if $\x_2$ is connected to $\p_1$ or $\p_2$, then $G$ is the graph in (1-b), and otherwise the graph in (1-c).

\emph{Subcase 1.2}: $\y$ is connected to $\p_1$ or $\p_2$. In this case, without loss of generality, we may assume that $\y$ is connected to $\p_2$ and it is not connected to $\p_1$. Using the same tools as in Subcase 1.1, we may obtain that $\x_2$ is connected to $\x_1$ and $\y$, and not connected to $\p_1$ and $\p_3$. Thus, depending on whether $\x_2$ is connected to $\p_2$ or not, $G$ is the graph in (1-a) or in (1-b), respectively.

\emph{Case 2}: $G$ contains $K_4$ as a subgraph.\\
First, we show that $\mathrm{card}(S) \leq 6$. For contradiction, set $V_1 := \{ \x_1, \y_1 \}$ and $V_2 := \{ \x_2, \y_2\}$.
Without loss of generality, we may assume that $\x_1$ is connected to all the $\p_i$s. This implies that neither $\x_2$ nor $\y_2$ is connected to at least three of the $\p_i$s and $\x_1$. Thus, without loss of generality, we may assume that $\x_2$ is connected to $\p_2$ and $\p_3$, and not to $\p_1$ and $\x_1$.
On the other hand, the fact that $\x_2$ and $\y_2$ are not connected implies that any of $\p_1, \p_2, \p_3, \x_1$ is connected to at least one of them, from which we have that $\y_2$ is connected to $\p_1$ and $\x_1$. But then $\{ \p_2, \p_3, \x_2 \}$ and $\{ \x_1, \y_2, \p_1\}$ are disjoint odd cycles, a contradiction (\cite{Do}).

Now we examine the case that $\mathrm{card}(S) = 6$. Let $V_1 := \{ \x_1, \y_1 \}$ and $V_2 := \{ \x_2\}$.
Consider the case that $\x_2$ is connected to all of the $\p_i$s, and let $T:= \{ \p_1, \p_2, \p_3, \x_2\}$. Since neither $\x_1$ nor $\y_1$ is connected to more than two points of $T$, and any point of $T$ is connected to at least one of $\x_1$ and $\y_1$, it follows that $\x_1$ is connected to exactly two points of $T$, say $\p_1$ and $\p_2$, and $\y_1$ is connected to the other two points. But then $\{ \x_1, \p_1, \p_2 \}$ and $\{ \y_1, \p_3, \x_2 \}$ are two disjoint odd cycles, contradicting \cite{Do}. Thus, we have that $\x_2$ is not connected to all vertices of $C$. This implies that one of $\x_1$ or $\y_1$, say $\x_1$, is connected to all $\p_i$s, and also that the remaining two points, $\y_1$ and $\x_2$, are connected. Thus, $G$ contains the graph in (2-a) as a subgraph. To investigate which other pairs of vertices can be connected or not, we may use the same tools as in Case 1, and hence, we omit it. This completes the proof of Part {\bf (v)}.

\section{Appendix}

Finally, for the sake of completeness, we call the reader's attention to Problem \ref{BNV-extended} below that investigates an interesting metric relative as well as extension of $f_k(\Md_{\K_{\oo}})$. In what follows, we give the relevant extension of Definition \ref{k-diametral} and a brief overview of the results that lead to Problem \ref{BNV-extended}.

\begin{definition}\label{k-equidistant}
Let $X:=\{\x_1,\x_2, \dots , \x_n\}\subset\Rd$ be a point configuration of $n$ points and $k\geq 2$ be an integer. We say that $X:=\{\x_1,\x_2, \dots , \x_n\}\subset\Rd$ is a {\rm $k$-equidistant point configuration of $n$ points in $\Md_{\K_{\oo}}$} if any $k$-tuple $\x_{n_1}, \x_{n_2}, \dots , \x_{n_k}$, $1\leq n_1< n_2 < \dots  < n_k\leq n$ chosen from $X$ contains two points lying at distance one in $\Md_{\K_{\oo}}$, i.e., there exist $\x_{n_i}$ and $\x_{n_j}$, $1\leq i<j\leq k$ such that $\| \x_{n_i}- \x_{n_j}\|_{\K_{\oo}} =1$. In particular, a $2$-equidistant point configuration is called a {\rm diametral} or simply an {\rm equilateral} point configuration, and a $3$-equidistant point configuration is called an {\rm almost equidistant} point configuration. Finally, let us denote the largest $n$ for which there exists a $k$-equidistant point configuration of $n$ points in $\Md_{\K_{\oo}}$, by $f^*_k(\Md_{\K_{\oo}})$ and call it the {\rm $k$-equidistant number} of point configurations in $\Md_{\K_{\oo}}$.
\end{definition}

\begin{remark}\label{basic-inequality}
Clearly, every $k$-diametral point configuration is a $k$-equidistant point configuration and therefore $f_k(\Md_{\K_{\oo}})\leq f^*_k(\Md_{\K_{\oo}})$ holds for all $d\geq 2$, $k\geq 2$, and $\K_{\oo}\in {\cal{K}}_{\oo}^{d}$.
\end{remark}

\begin{remark}\label{Bezdek-Naszodi-Visy}
We note that Bezdek, Nasz\'odi, and Visy \cite{BNV} introduced $f^*_k(\Md_{\K_{\oo}})$ under the name {\rm $k$th Petty number of $\Md_{\K_{\oo}}$} and defined it in the same way as Definition~\ref{k-equidistant}, but only for point configurations consisting of distinct points.
As the volumetric methods of \cite{BNV} extend to point configurations in general, therefore the upper bounds proved for $k$th Petty numbers in \cite{BNV} are upper bounds for $f^*_k(\Md_{\K_{\oo}})$ as well and so, the inequalities
\begin{equation}\label{BNV-1}
f^*_3(\Ee^2)=7, f^*_3(\Mm^2_{\K_{\oo}})\leq 8,\ {\text and}\ f^*_k(\Mm^2_{\K_{\oo}})\leq 8(k-1)
\end{equation}
\begin{equation}\label{BNV-2}
f^*_k(\Ed)\leq (k-1)3^d, f^*_3(\Md_{\K_{\oo}})\leq 2\cdot3^d, \ {\text and}\  f^*_k(\Md_{\K_{\oo}})\leq\min\left\{(k-1)4^d, (k-1)\left((k-1)3^d-(k-2)\right)\right\}
\end{equation}
hold for all $k\geq 4$, $d\geq 3$ and $\K_{\oo}\in {\cal{K}}_{\oo}^{d}$.
\end{remark}

\begin{remark}\label{KMS-P-result}
Recall that Kupavskii, Mustafa, and Swanepoel \cite{KMS} and also Polyanskii \cite{Po} introduced  $f^*_3(\Ed)$ the same way as Definition~\ref{k-equidistant}, but only for point configurations consisting of distinct points. As the elegant algebraic methods of \cite{KMS} as well as \cite{Po} extend to point configurations of $\Ed$, therefore the upper bound $O(d^{\frac{4}{3}})$ proved in \cite{KMS} as well as \cite{Po} works for $f^*_3(\Ed)$ as well. 
\end{remark}

\noindent Now, we are ready to state the open question which is an extension of the Problem of \cite{BNV} for point configurations in general.

\begin{problem}\label{BNV-extended}
Prove or disprove that $f^*_k(\Md_{\K_{\oo}})\leq (k-1)2^d$ holds for all $k\geq 2$, $d\geq 2$ and $\K_{\oo}\in {\cal{K}}_{\oo}^{d}$.
\end{problem}

\begin{remark}\label{BNV-special}
As the method of proof of Theorem 3 in \cite{BNV} extends to point configurations in general, therefore if $\K_{\oo}$ is an $\oo$-symmetric affine $d$-cube of $\Rd$, then $f^*_k(\Md_{\K_{\oo}})=(k-1)2^d$ holds for all $k\geq 2$, $d\geq 2$. 
\end{remark}

\bigskip


\noindent K\'aroly Bezdek \\
\small{Department of Mathematics and Statistics, University of Calgary, Canada}\\
\small{Department of Mathematics, University of Pannonia, Veszpr\'em, Hungary\\
\small{E-mail: \texttt{bezdek@math.ucalgary.ca}}

\bigskip

\noindent and

\bigskip

\noindent Zsolt L\'angi \\
\small{MTA-BME Morphodynamics Research Group and Department of Geometry}\\ 
\small{Budapest University of Technology and Economics, Budapest, Hungary}\\
\small{\texttt{zlangi@math.bme.hu}}

\end{document}